\documentclass[10pt,a4paper,reqno]{amsart}
\usepackage{amsmath,amsthm,amssymb,
	latexsym,mathrsfs,amsfonts,cases,indentfirst,bbold}
\usepackage{color}
\usepackage[leqno]{amsmath} 
\usepackage{amsthm,amsmath,amssymb}
\usepackage{mathrsfs}
\usepackage{amssymb}
\usepackage{amsfonts}
\usepackage{latexsym}
\usepackage{ulem}
\usepackage{geometry}
\usepackage{hyperref}
\usepackage{graphicx}
\usepackage{subfigure}
\usepackage{cases}
\usepackage{array}
\usepackage{color}
\usepackage{fancyhdr}
\usepackage[all,pdf]{xy}
\usepackage{tikz}       
\usepackage[all]{xy}	
\usepackage{amsmath,amssymb}

\usepackage{amsfonts}

\usepackage{color}      
\usepackage{indentfirst}        
\usepackage{latexsym, bm}        
\usepackage{graphicx}
\usepackage{cases}
\usepackage[all]{xypic}
\usepackage{fancyhdr}
\usepackage{pifont}
\newtheorem{theorem}{Theorem}
\newtheorem{defi}[theorem]{Definition}
\newtheorem{lemma}[theorem]{Lemma}
\newtheorem{coro}[theorem]{Corollary}
\newtheorem{proposition}[theorem]{Proposition}
\newtheorem{example}[theorem]{Example}

\newtheorem{remark}[theorem]{Remark}
\usepackage[all]{xy}
\usepackage{graphicx}
\usepackage{epstopdf}

\begin{document}

	\title[A quantum shuffle approach to $U_q(C(2)^{(2)})$ and equitable presentation]{A quantum shuffle approach to quantum affine super algebra of type  $C(2)^{(2)}$ and its equitable presentation}
	
	\author[Xin Zhong]{Xin Zhong}
\address{School of Mathematical Sciences,  East China Normal University, Shanghai 200241, China\\
Current Address: No. 2 High School of East China Normal University, Yangtai Road 500, Baoshan, China}
	\email{52215500002@stu.ecnu.edu.cn}
	
	\author[N.H. Hu]{Naihong Hu$^*$}
\address{School of Mathematical Sciences, MOE Key Laboratory of Mathematics and Engineering Applications \& Shanghai Key Laboratory of PMMP, East China Normal University, Shanghai 200241, China}
	\email{nhhu@math.ecnu.edu.cn}
	\thanks{$^*$Corresponding author.
	The paper is supported by the NNSFC (Grant No.
	12171155) and in part by the Science and Technology Commission of Shanghai Municipality (No. 22DZ2229014).}
	\begin{abstract}
In this study, we focus on the positive part $U_q^{+}$ of the quantum affine superalgebra $U_q(C(2)^{(2)})$. This algebra admits a presentation with two two generators $e_{\alpha}$ and $e_{\delta-\alpha}$, which satisfy the cubic $q$-Serre relations.
According to the work of  Khoroshkin-Lukierski-Tolstoy, the Damiani and the Beck $PBW$ bases  exist for this superalgebra. In this paper, we utilize the $q$-shuffle superalgebra and Catalan words to present these two bases in a closed-form expression.
Ultimately, we present the bosonization of $U_q(C(2)^{(2)})$.
	\end{abstract}

	\maketitle
		\section{Introduction}
		\subsection{Background and Motivation}
		The $q$-deformed universal enveloping algebra $U_q(\widehat{\mathfrak{s l}_2})$ (\cite{KT}, \cite{CP}) appears in the topics of combinatorics, quantum algebras, and representation theory. The positive part $U_q^{+}$ of $U_q(\widehat{\mathfrak{s l}_2})$ is associative, noncommutative, and infinite-dimensional. It has a presentation with two generators and relations including the so-called $q$-Serre relations. The literature contains at least three $PBW$ bases for $U_q^{+}$, called the Damiani, the Beck, and the alternating $PBW$ bases. These $PBW$ bases are related via exponential formulas and quantum shuffle algebra (embedding the ``positive" subalgebra into the $q$-deformed shuffle algebra) was first developed independently in the 1990 by J. Green, and M. Rosso \cite{R2, G}.
		
		In \cite{D}, Damiani obtained a Poincaré-Birkhoff-Witt (or $PBW$) basis for $U_q^{+}=U_q^+(\widehat{\mathfrak{sl}_2})$. The basis elements $\left\{E_{n \delta+\alpha_0}\right\}_{n=0}^{\infty}$, $\left\{E_{n \delta+\alpha_1}\right\}_{n=0}^{\infty}$, $\left\{E_{n \delta}\right\}_{n=1}^{\infty}$ were defined recursively using a braid group action. In \cite{B}, Beck obtained a $PBW$ basis for $U_q^{+}$ by replacing $E_{n \delta}$ with an element $E_{n \delta}^{\text {Beck }}$ for $n \geq 1$. In \cite{BCP}, Beck, Chari, and Pressley showed that the elements $\left\{E_{n \delta}\right\}_{n=1}^{\infty}$ and $\left\{E_{n \delta}^{\text {Beck }}\right\}_{n=1}^{\infty}$ are related via an exponential formula.
		
		Based on the above work, Terwilliger used the Rosso embedding to obtain a closed form for the Damiani and the Beck $PBW$ basis elements as follows \cite{T2,T3}
\begin{gather*}
		E_{n \delta+\alpha_0} \mapsto q^{-2 n}\left(q-q^{-1}\right)^{2 n} x C_n, \quad \quad E_{n \delta+\alpha_1} \mapsto q^{-2 n}\left(q-q^{-1}\right)^{2 n} C_n y,
\end{gather*}
		for $n \geq 0$, and
\begin{gather*}
		E_{n \delta} \mapsto-q^{-2 n}\left(q-q^{-1}\right)^{2 n-1} C_n,\\
		E_{n \delta}^{\text {Beck }} \mapsto \frac{[2 n] q}{n} q^{-2 n}\left(q-q^{-1}\right)^{2 n-1} x C_{n-1} y,
\end{gather*}
		for $n \geq 1$.
		
		In \cite{T4}, Terwilliger introduced alternating words and obtained some $PBW$ bases. For instance, the elements $\left\{W_{-k}\right\}_{k=0}^{\infty},\left\{W_{k+1}\right\}_{k=0}^{\infty},\left\{\tilde{G}_{k+1}\right\}_{k=0}^{\infty}$ (in appropriate linear order) give a $PBW$ basis said to be alternating. Then he interpreted the alternating $PBW$ basis in terms of a $q$-shuffle algebra associated with quantum affine algebra $U_q(\widehat{\mathfrak{s l}_2})$ and showed how the alternating $PBW$ basis is related to the $PBW$ basis for $U_q^{+}$ found by Damiani in 1993.
	
	In \cite{ITW}, Ito, Terwilliger, and Weng introduced the equitable presentation for the quantum group $U_q(\mathfrak{sl}_2)$. This presentation is linked to tridiagonal pairs of linear transformations \cite{IT1,IT2}, Leonard pairs of linear transformations \cite{SGH}. The equitable presentation has been generalized to the  quantum group $U_q(\mathfrak{g})$ associated to symmetrizable Kac-Moody algebra $\mathfrak{g}$ \cite{T5}. Furthermore, The equitable presentation of the quantum superalgebra $\mathfrak{osp}_q(1|2)$, in which all generators appear on an equal footing, was exhibited, and a $q$-analog of the Bannai-Ito algebra is shown to arise as the covariance algebra of $\mathfrak{osp}_q(1|2)$ \cite{GVZ}. In \cite{SL}, Sun and Li provided the equitable presentation for a certain subalgebra of the two-parameter quantum group $U_{r, s}(\mathfrak{g})$ associated with a generalized Kac-Moody algebra corresponding to symmetrizable admissible Borcherds Cartan matrix. More generally, Hu-Pei-Zhang  \cite{HPZ} gave an equitable presentation in the sense of Terwilliger for the multiparameter quantum groups $U_{\mathbf{q}}(\mathfrak{g})$ associated to symmetrizable Kac-Moody algebras $\mathfrak{g}$ and multiparameter matrices $\mathbf{q}=\left(q_{i j}\right)$, which turn out to be constructed from the one-parameter quantum groups by twisting the product by Hopf $2$-cocycles \cite{HPZ}.

Consequently, a natural query arises: how can we extend the existing research to quantum affine superalgebras, such as $U_q(C(2)^{(2)})$ (\cite{KLT}).  This is a technically intricate yet fascinating open problem. In this paper and a preprint \cite{ZH2}, we aim to complete the relevant work for the quantum affine superalgebra $U_q(C(2)^{(2)})$.
		\subsection{Goal} Throughout the paper, we will focus on the quantum affine superalgebra $U_q(C(2)^{(2)})$ presented in \cite{KLT}.

The first goal of this paper is to establish an injective superalgebra homomorphism $\varphi $ from $U_q^{+}(C(2)^{(2)})$ to the $q$-shuffle superalgebra $\mathbb{V}$, that maps $A \mapsto x$ and $B \mapsto y$; this result is a direct consequence of Proposition \ref{shuffle serre}. Our subsequent goal is to express the KLT $PBW$ basis of Damiani type and Beck type in a closed form by utilizing Catalan words and the $q$-shuffle superalgebra. As direct corollaries thereof, we derive several relations involving the higher Catalan elements. Finally, we achieve a minimal bosonization of $U_q(C(2)^{(2)})$.

		\subsection{Further Applications} (1) The $q$-shuffle superalgebra and bosonization of $U_q(C(2)^{(2)})$ have played a fundamental role in guiding our work. In our recent work \cite{ZH1}, with the help of the above results, we discovered a new $q$-Onsager superalgebra $\mathcal{B}_c^s$ associated with  $U_q(C(2)^{(2)})$ (actually, this algebra is a new $i$-quantum group of rank one in  super world), which is generated by $B_0, B_1$ and relations
		$$
		\begin{aligned}
			\sum_{i=0}^3(-1)^i\left\{\begin{array}{l}
				3 \\
				i
			\end{array}\right\}_q B_0^{3-i} B_1 B_0^i & =q c\left(q-q^{-1}\right)^2\left(B_0 B_1-B_1 B_0\right), \\
			\sum_{i=0}^3(-1)^i\left\{\begin{array}{l}
				3 \\
				i
			\end{array}\right\}_q B_1^{3-i} B_0 B_1^i & =q c\left(q-q^{-1}\right)^2\left(B_1 B_0-B_0 B_1\right).
		\end{aligned}
		$$
	 And inspired by Damiani’s construction and investigation of root vectors, we can get $PBW$ bases for  $\mathcal{B}_c^s$. Furthermore, we can obtain the Drinfeld second realization of $\mathcal{B}_c^s$ in our subsequent work.
	
	 (2) The positive part $U_q^{+}$ of the quantum affine superalgebra $U_q(C(2)^{(2)})$ can be presented using two generators, $A$ and $B$, which satisfy the cubic $q$-Serre relations. We interpret the alternating Poincar\'e-Birkhoff-Witt basis in the context of a $q$-shuffle superalgebra that we have introduced in association with $U_q(C(2)^{(2)})$. Moreover, we demonstrate the relationship between the alternating $PBW$ basis and the KLT $PBW$ basis for $U_q^{+}$.
Furthermore, we introduce an algebra $\mathcal{U}_q^{+}$, which is referred to as the minimal bosonization of $U_q^{+}$. While its presentation is appealing, the large number of generators and relations renders it cumbersome. Subsequently, we manage to obtain a presentation of $\mathcal{U}_q^{+}$ that involves only a small subset of the original generators and a highly tractable set of relations. We term this a compact presentation \cite{ZH2}.

\subsection{Advantage}
In the context of our research, it is crucial to define these Catalan words for $U_q(C(2)^{(2)})$ recursively, rather than directly as originally proposed in \cite{T2, T3} for $U_q(\widehat{\mathfrak{s l}_2})$. The root cause is that the shuffle relations within our superalgebraic context are considerably more complex than those in $U_q(\widehat{\mathfrak{sl}_2})$. Notably, the recursive definition we adopt provides a unique benefit. When computing higher-order shuffles, we can recursively utilize lower-order shuffles. Leveraging this definition, we can efficiently establish Theorems \ref{Main thm 1} and \ref{Main thm 2} without resorting to Terwilliger’s so-called Catalan profile method.

\section{quantum affine superalgebra $U_q(C(2)^{(2)})$ and $q$-shuffle superalgebra}
\subsection{ Quantum Affine Superalgebra $U_q(C(2)^{(2)})$ } Prior to delving into the details, let us first take a moment to establish certain notations and definitions.
\begin{defi} $($\cite{KLT}$)$
	The quantum affine superalgebra $U_q(C(2)^{(2)})$  is generated by
	the Chevalley elements $k_{\mathrm{d}}^{ \pm 1}:=q^{ \pm h_{\mathrm{d}}}, k_\alpha^{ \pm 1}:=q^{ \pm h_\alpha}, k_{\delta-\alpha}^{ \pm 1}:=q^{ \pm h_{\delta-\alpha}}, e_{ \pm \alpha}, e_{ \pm(\delta-\alpha)}$ with the defining relations
	\begin{gather}
		k_\gamma k_\gamma^{-1}=k_\gamma^{-1} k_\gamma=1, \quad\left[k_\gamma^{ \pm 1}, k_{\gamma^{\prime}}^{ \pm 1}\right]=0,\\
		k_\gamma e_{ \pm \alpha} k_\gamma^{-1}=q^{ \pm(\gamma, \alpha)} e_{ \pm \alpha}, \quad k_\gamma e_{ \pm(\delta-\alpha)} k_\gamma^{-1}=q^{ \pm(\gamma, \delta-\alpha)} e_{ \pm(\delta-\alpha)},\\
		\left[e_\alpha, e_{-\alpha}\right]=\frac{k_\alpha-k_\alpha^{-1}}{q+q^{-1}}, \quad\left[e_{\delta-\alpha}, e_{-\delta+\alpha}\right]=\frac{k_{\delta-\alpha}-k_{\delta-\alpha}^{-1}}{q+q^{-1}},\\
		\left[e_\alpha, e_{-\delta+\alpha}\right]=0, \quad\left[e_{-\alpha}, e_{\delta-\alpha}\right]=0,\\
		\label{serre 1}
		\left[e_{ \pm \alpha},\left[e_{ \pm \alpha},\left[e_{ \pm \alpha}, e_{ \pm(\delta-\alpha)}\right]_q\right]_q\right]_q=0,\\
		\label{serre 2}
		\left[\left[\left[e_{ \pm \alpha}, e_{ \pm(\delta-\alpha)}\right]_q, e_{ \pm(\delta-\alpha)}\right]_q, e_{ \pm(\delta-\alpha)}\right]_q=0,
	\end{gather}
	where $\gamma \in \{\mathrm{d}, \alpha, \delta-\alpha\}$ (with $\mathrm{d}$ the degree-derivation element of affine Lie superalgebra $C(2)^{(2)}$, $\alpha$ the finite simple root of the underlying superalgebra $C(2)$, and $\delta$ the null root of $C(2)^{(2)}$),
	$(\mathrm{d}, \alpha)=0$, $(\mathrm{d}, \delta)=1$,
	The brackets $[\cdot, \cdot]$, $[\cdot, \cdot]_q$ are the super-commutators and $q$-super-commutators
	\begin{gather*}
		\left[e_\beta, e_{\beta^{\prime}}\right]=e_\beta e_{\beta^{\prime}}-(-1)^{\vartheta(\beta) \vartheta\left(\beta^{\prime}\right)} e_{\beta^{\prime}} e_\beta,\\
		\left[e_\beta, e_{\beta^{\prime}}\right]_q=e_\beta e_{\beta^{\prime}}-(-1)^{\vartheta(\beta) \vartheta\left(\beta^{\prime}\right)} q^{\left(\beta, \beta^{\prime}\right)} e_{\beta^{\prime}} e_\beta,
	\end{gather*}
	where the symbol $\vartheta(\cdot)$ denotes the parity function on the root lattice of $C(2)^{(2)}$: $\vartheta(\beta)=0$ for any even root $\beta$ (element of the even part of the superalgebra), and $\vartheta(\beta)=1$ for any odd root $\beta$ (element of the odd part of the superalgebra).
\end{defi}
\begin{remark}
 The left sides of the relations $($\ref{serre 1}$)$ and $($\ref{serre 2}$)$  are invariant with respect to the
	replacement of $q$ by $q^{-1}$. In fact, when we eliminate the $q$-brackets, it becomes evident that the left-hand sides of (\ref{serre 1}) and (\ref{serre 2}) consist of symmetric functions of  $q$ and $q^{-1}$. This characteristic enables us to rewrite the  $q$-commutators in relations (\ref{serre 1}) and (\ref{serre 2}) in the reverse order,
	$$\left[\left[\left[e_{ \pm(\delta-\alpha)}, e_{ \pm \alpha}\right]_q, e_{ \pm \alpha}\right]_q, e_{ \pm \alpha}\right]_q=0,$$
	$$
	\left[e_{ \pm(\delta-\alpha)},\left[e_{ \pm(\delta-\alpha)},\left[e_{ \pm(\delta-\alpha)}, e_{ \pm \alpha}\right]_q\right]_q\right]_q=0.
	$$
	\end{remark}
Denote by  $U^{+}_q(C(2)^{(2)})$ the positive part of $U_q(C(2)^{(2)})$ and define the normal ordering. As we know, the reduced positive system has only two normal orderings as follows
$$
\begin{aligned}
	& \alpha, \delta+\alpha, 2 \delta+\alpha, \ldots, \infty \delta+\alpha, \delta, 2 \delta, 3 \delta, \ldots, \infty \delta, \infty \delta-\alpha, \ldots, 3 \delta-\alpha, 2 \delta-\alpha, \delta-\alpha, \\
	& \delta-\alpha, 2 \delta-\alpha, 3 \delta-\alpha, \ldots, \infty \delta-\alpha, \delta, 2 \delta, 3 \delta, \ldots, \infty \delta, \infty \delta+\alpha, \ldots, 2 \delta+\alpha, \delta+\alpha, \alpha,
\end{aligned}
$$
where the first normal ordering corresponds to a “clockwise” arrangement for the positive roots. Specifically, this arrangement starts from the root $\alpha$  and progresses to the root $\delta-\alpha$. Conversely, the inverse normal ordering corresponds to an “anticlockwise” arrangement of the positive roots. In this case, the movement is from $\delta-\alpha$ to $\alpha$.

For $U_q(C(2)^{(2)})$, we require the following   simple (anti)automorphisms.
\begin{proposition}$($\cite{KLT}$)$
	There exists a graded antilinear antiinvolution  `` $\ddagger$ ":
$$
\begin{aligned}
	\left(q^{ \pm 1}\right)^{\ddagger} & =q^{\mp 1}, & \left(k_\gamma^{ \pm 1}\right)^{\ddagger} & =k_\gamma^{\mp 1}, \\
	E_\beta^{\ddagger} & =(-1)^{\vartheta(\beta)} E_{-\beta}, & E_{-\beta}^{\ddagger} & =E_\beta.
\end{aligned}
$$
	where $(x y)^{\ddagger}=(-1)^{\operatorname{deg} x \operatorname{deg} y} y^{\ddagger} x^{\ddagger}$ for any homogeneous elements $x, y \in U_q(g)$.
	\end{proposition}
\begin{proposition} $($\cite{KLT}$)$
There exists a Dynkin involution $\tau$ which is associated with the automorphism of the Dynkin diagrams of the $U_q\left(C(2)^{(2)}\right)$:
$$
\begin{aligned}
	\tau\left(q^{ \pm 1}\right) & =q^{ \pm 1}, & & \tau\left(k_{\mathrm{d}}^{ \pm 1}\right)=k_{\mathrm{d}}^{ \pm 1}, \\
	\tau\left(k_\beta^{ \pm 1}\right) & =k_{\delta-\beta}^{ \pm 1}, & & \tau\left(k_{-\beta}^{ \pm 1}\right)=k_{-\delta+\beta}^{ \pm 1}, \\
	\tau\left(E_\beta\right) & =E_{\delta-\beta}, & & \tau\left(E_{-\beta}\right)=E_{-\delta+\beta} .
\end{aligned}
$$
	\end{proposition}

  In accordance with the second normal ordering, and setting $\delta-\alpha$ to be $\alpha_0$ and $\alpha$ to $\alpha_1$,  we can define the
  PBW generators of Damiani type as follows:
\begin{gather}
\left\{E_{n \delta+\alpha_0}\right\}_{n=0}^{\infty}, \quad\left\{E_{n \delta+\alpha_1}\right\}_{n=0}^{\infty}, \quad\left\{E_{n \delta}\right\}_{n=1}^{\infty}.
\label{root vec}
\end{gather}
  These elements are recursively defined as
\begin{gather}
  E_{\alpha_0}=A, \quad E_{\alpha_1}=B, \quad E_\delta=q^{-2} B A+A B,
  \label{simple vec}
\end{gather}
  and for $n \geq 1$,
\begin{gather}
	\label{real root}
E_{n \delta+\alpha_0}=\frac{\left[E_{(n-1) \delta+\alpha_0}, E_\delta\right]}{q-q^{-1}}, \quad E_{n \delta+\alpha_1}=\frac{\left[E_\delta, E_{(n-1) \delta+\alpha_1}\right]}{q-q^{-1}}\\
  \label{im root}
  E_{n \delta}=q^{-2} E_{(n-1) \delta+\alpha_1} A+A E_{(n-1) \delta+\alpha_1},
\end{gather}
  where
\begin{gather*}
  \left[E_\beta, E_{\beta^{\prime}}\right]=E_\beta E_{\beta^{\prime}}-(-1)^{\vartheta(\beta) \vartheta\left(\beta^{\prime}\right)} E_{\beta^{\prime}} E_\beta,\\
  \left[E_\beta, E_{\beta^{\prime}}\right]_q=E_\beta E_{\beta^{\prime}}-(-1)^{\vartheta(\beta) \vartheta\left(\beta^{\prime}\right)} q^{\left(\beta, \beta^{\prime}\right)} E_{\beta^{\prime}} E_\beta.
\end{gather*}
And by \cite{KLT}, the elements $\left\{E_{n \delta}\right\}_{n=1}^{\infty}$ mutually commute. Based on the authors' work in \cite{KLT}, we are able to prove the validity of the following Theorem.
\begin{theorem} \label{PBW Da}
	Define
	$$
			\mathscr{B ^ { + }} =\Big\{\, E_{\beta_1}^{r_1} \cdot \ldots \cdot E_{\beta_N}^{r_N} \,\Big|\, N \in \mathbb{N},\ \beta_1<\ldots<\beta_N \in \mathscr{R}_{+},\
		r_1, \ldots, r_N \in \mathbb{N} \backslash\{0\}\,\Big\}.
	$$
	Then, the set $\mathscr{B}^{+}$ defined above is a $PBW$ basis of $U_q^{+}\left(C(2)^{(2)}\right)$. Moreover, if we define
$$F_\alpha = \ddagger\left(E_\alpha\right), \quad \forall\, \alpha \in \mathscr{R}_{+},$$
then the vectors of the form $F_{\gamma_1}^{s_1} \cdot \ldots \cdot F_{\gamma_M}^{s_M} K_1^n K_2^m E_{\beta_1}^{r_1} \cdot \ldots \cdot E_{\beta_N}^{r_N}$ with $M, N \in \mathbb{N}$, $\gamma_1>\ldots>\gamma_M, \beta_1<\ldots<\beta_N \in \mathscr{R}_{+}, s_1, \ldots, s_M, r_1, \ldots, r_N \in \mathbb{N} \backslash\{0\}, m, n \in \mathbb{Z}$ form a $PBW$ basis of $U_q\left(C(2)^{(2)}\right)$.
	
	\end{theorem}
\subsection{$q$-shuffle superalgebra}
We fix two noncommuting indeterminates $x$ and $y$, and let $\mathbb{V}$ denote the free associative algebra over these generators. For each $n \in \mathbb{N}$, a word of length $n$ in $\mathbb{V}$ is a product $v_1 v_2 \cdots v_n$ with $v_i \in \{x, y\}$ for all indices $1 \leq i \leq n$. By convention, the empty product (corresponding to length zero) is taken as the multiplicative identity of $\mathbb{V}$; this element is called the trivial word and written as $1$. The vector space underlying $\mathbb{V}$ admits a basis composed of all words in $x$ and $y$; this basis is termed the standard basis.
\begin{defi}\label{orth}
	For vector space $\mathbb{V}$, there exists a unique bilinear form $( $,$ ): \mathbb{V} \times \mathbb{V} \rightarrow \mathbb{F}$ with respect to which the standard basis is orthonormal. The bilinear form $( $,$ )$ is symmetric and nondegenerate. For $v \in \mathbb{V}$, we have
	$$
	v=\sum_w w(w, v),
	$$
	where the summation is over all the words $w$ in $\mathbb{V}$.
\end{defi}
The $q$-shuffle algebra was introduced by Rosso \cite{R1, R2} and described further by Green \cite{G}. Within the framework of the superalgebra $U_q(C(2)^{(2)})$, we need to define the $q$-shuffle superproduct for nontrivial words as follows.

\begin{defi}
On the free superalgebra $\mathbb{V}$, there exists a $q$-shuffle superproduct, denoted by $\star$. It can be defined as follows, endowing  with the structure of a $q$-shuffle superalgebra.

For $v \in \mathbb{V}$, $1 \star v=v \star 1=v$.  For $x, y \in \mathbb{V}$ with $\text{deg}(x)=\text{deg}(y)=1$,
\begin{gather*}
	x \star y=x y-q^{-2} y x, \quad y \star x=y x-q^{-2} x y,\\
	x \star x=\left(1-q^2\right) x x, \quad y \star y=\left(1-q^2\right) y y.
\end{gather*}

 For $u \in\{x, y\}$ and a nontrivial word $v=v_1 v_2 \cdots v_n$ in  $\mathbb{V}$, define the $q$-shuffle superprodcut
 \begin{gather}
	u \star v=\sum_{i=0}^n (-1)^{i}v_1 \cdots v_i u v_{i+1} \cdots v_n q^{\left\langle v_1, u\right\rangle+\left\langle v_2, u\right\rangle+\cdots+\left\langle v_i, u\right\rangle},
	\label{shuffle long 1}
\end{gather}
\begin{gather}
	v \star u=\sum_{i=0}^n (-1)^{n-i} v_1 \cdots v_i u v_{i+1} \cdots v_n q^{\left\langle v_n, u\right\rangle+\left\langle v_{n-1}, u\right\rangle+\cdots+\left\langle v_{i+1}, u\right\rangle},
	\label{shuffle long 2}
\end{gather}
where
$$
\begin{array}{c|cc}
	\langle,\rangle & x & y \\
	\hline x & 2 & -2 \\
	y & -2 & 2
\end{array}
$$	

For nontrivial words $u=u_1 u_2 \cdots u_r$ and $v=v_1 v_2 \cdots v_s$ in $\mathbb{V}$, we define
\begin{gather}
	\label{shuffle long 3}
	u \star v=u_1\left(\left(u_2 \cdots u_r\right) \star v\right)+(-1)^{r}v_1\left(u \star\left(v_2 \cdots v_s\right)\right) q^{\left\langle u_1, v_1\right\rangle+\left\langle u_2, v_1\right\rangle+\cdots+\left\langle u_r, v_1\right\rangle},\\
	\label{shuffle long 4}
	u \star v=\left(u \star\left(v_1 \cdots v_{s-1}\right)\right) v_s+(-1)^{s}\left(\left(u_1 \cdots u_{r-1}\right) \star v\right) u_r q^{\left\langle u_r, v_1\right\rangle+\left\langle u_r, v_2\right\rangle+\cdots+\left\langle u_r, v_s\right\rangle}.
\end{gather}
	\end{defi}
\begin{remark}
When compared with the $q$-shuffle algebra associated with $U_q(\widehat{\mathfrak{s l}_2})$, the $q$-shuffle superproduct associated with $U_q(C(2)^{(2)})$ has an extra sign deformation, which is related to the parity of the total root system $\Delta$ of $U_q(C(2)^{(2)})$.
\end{remark}
\begin{example}

$$
\begin{gathered}
	x \star(y y y)=x y y y-q^{-2} y x y y+q^{-4} y y x y-q^{-6} y y y x, \\
	(x y x) \star y=x y x y+\left(1-q^{-2}\right) x y y x-q^{-2} y x y x , \\
	(x x) \star(y y y)=x x y y y-q^{-2} x y x y y+q^{-4} x y y x y-q^{-6} x y y y x+q^{-4} y x x y y \\
	-q^{-6} y x y x y+q^{-8} y x y y x+q^{-8} y y x x y-q^{-10} y y x y x+q^{-12} y y y x x, \\
	(x y) \star(x x y y)=x y x x y y+x x y y x y-\{2\}_q^2 x x y x y y+\{3\}_q^2 x x x y y y,
\end{gathered}
$$
where $\{n\}_q=\frac{q^{-n}-(-1)^n q^n}{q+q^{-1}}$.	
	\end{example}

\subsection{Connection bewteen $U_q^{+}(C(2)^{(2)})$ and $q$-shuffle superalgebra \uppercase\expandafter{\romannumeral1} }
In this subsection, we will get the correspondence  between Khoroshkin-Lukierski-Tolstoy (KLT)'s $PBW$ basis of Damiani type and Catalan words.
\begin{proposition}\label{shuffle serre}
	The following equations hold in $\mathbb{V}$,
\begin{align}
		\label{serre shuffle 1}
	x \star x \star x \star y+\{3\}_q x \star x \star y \star x-\{3\}_q x \star y \star x \star x-y \star x \star x \star x=0,\\
	\label{serre shuffle 2}
	y \star y \star y \star x+\{3\}_q y \star y \star x \star y-\{3\}_q y \star x \star y \star y-x \star y \star y \star y=0 .
\end{align}
\end{proposition}
	\begin{proof}
		We only prove (\ref{serre shuffle 1}). (\ref{serre shuffle 2}) is similar. Thanks to (\ref{shuffle long 1}), (\ref{shuffle long 2}), (\ref{shuffle long 3}), (\ref{shuffle long 4}), we can get
\begin{gather*}
		x \star x \star x \star y=\left(1-q^2\right)\left(1-q^2+q^4\right)\left(x^3 y-q^{-2} x^2 y x+q^{-4} x y x^2-q^{-6} y x^3\right),\\
		x \star x \star y \star x=(1{-}q^2)\left[(1{-}q^{-2}{-}q^2) x^3 y{+}(2{-}q^{-2}) x^2 y x{+}(1{-}2 q^{-2}) x y x^2{+}q^{-4}(1{-}q^2{+}q^4) y x^3\right],\\
		x \star y \star x \star x=(1{-}q^{-2})\left[(1{-}q^{-2}{-}q^2) x^3 y{+}(2{-}q^2) x^2 y x{+}(1{-}2 q^2) x y x^2{+}(1{-}q^2{+}q^4) y x^3\right],\\
		y \star x \star x \star x=(q^{-4}{-}q^{-2})(1{-}q^{-2}{-}q^2) x^3 y{+}(1{-}q^{-2})(1{-}q^{-2}{-}q^2) x^2 y x\\
		\qquad\qquad +\,(1{-}q^2)(1{-}q^{-2}{-}q^2) x y x^2+(1{-}q^2)(1{-}q^2{+}q^4) y x^3.
	\end{gather*}
		After a simple calculation, we get (\ref{serre shuffle 1}).
\end{proof}
\begin{coro}
	There exists an injective superalgebra homomorphism $\varphi $ from $U_q^{+}(C(2)^{(2)})$ to  the $q$-shuffle super-algebra $\mathbb{V}$, that sends $A \mapsto x$ and $B \mapsto y$.
\end{coro}
\begin{proof}
	The superalgebra homomorphism follows directly from Proposition \ref{shuffle serre}. Therefore, it suffices to prove that $\varphi $ is injective.
	
	A bilinear form $(\cdot,\cdot): U_q^+ \times U_q^+ \to \mathbb{Q}(q)$ is defined on $U_q^+$, satisfying the following properties:
\begin{itemize}
	\item Generator pairing:
For the generators $A$ and $B$ of $U_q^+$, we have:
\[
(A, A) = 1, \quad (B, B) = 1, \quad (A, B) = 0.
\]
	
	\item Multiplicative compatibility:
	For any $a, b, c \in U_q^+$, we have:
	\[
	(ab, c) = (a \otimes b, \Delta(c)),
	\]
where the tensor product pairing is defined by $(a \otimes b, c \otimes d) = (a, c)(b, d)$.
	
	\item Supersymmetry:
	For any homogeneous elements $a, b \in U_q^+$,
	\[
	(a, b) = (-1)^{\vartheta(a)\vartheta(b)} (b, a).
	\]
\end{itemize}
		By Theorem \ref{PBW Da}, the bilinear form $(\cdot, \cdot)$ is easily shown to be non-degenerate.
	
We define an algebra homomorphism $\xi: \mathbb{V} \to U_q^+$ by its action on the generators of $\mathbb{V}$:
\[
\xi(x) = A , \quad \xi(y) = B .
\]
As $U_q^+$ is generated by $A$ and $B$, $\xi$ is surjective.

By Definition \ref{orth}, for any $u \in U_q^+$, we can extend  $\varphi: U_q^+ \to \mathbb{V}$ by
\[
\varphi(u) = \sum_{w \in \mathbb{V}} (u, \xi(w)) w.
\]
We directly obtain the duality property:
\[
(\varphi(u), v) = (u, \xi(v)) \quad \forall\; u \in U_q^+, v \in \mathbb{V}.
\]

Suppose there exists $a \in U_q^+$ such that $\varphi(a) = 0$. Then for any $v \in \mathbb{V}$, we have:
\[
(a, \xi(v)) = (\varphi(a), v) = (0, v) = 0.
\]
Since $\xi$ is surjective, for any $b \in U_q^+$, there exists $v \in \mathbb{V}$ such that $\xi(v) = b$. Therefore:
\[
(a, b) = 0 \quad \forall\; b \in U_q^+.
\]
Furthermore, since the bilinear form $(\cdot, \cdot)$ is non-degenerate, it follows that $a = 0$, i.e., $\ker \varphi = \{0\}$.

This completes the proof.	
\end{proof}
 We still refer to it as the Rosso embedding for the superalgebra case.

Before stating our main Theorem, we need to introduce an important combinatorial tool --- Catalan words, which shall help us to establish the connection between $U_q^{+}(C(2)^{(2)})$ and the $q$-shuffle superalgebra.
  \begin{defi}
  	Define $\bar{x}=1$ and $\bar{y}=-1$. Pick an integer $n \geq 0$  and consider a word $w=a_1 a_2 \cdots a_n$ in  $\mathbb{V}$. The word $w$ is called balanced whenever $\bar{a}_1+\bar{a}_2+\cdots+\bar{a}_n=0 ;$ in this case $n$ is even. The word $w$ is said to be $Catalan$, whenever it is balanced and $\bar{a}_1+\bar{a}_2+\cdots+\bar{a}_i \geq 0$ for $1 \leq i \leq n$.
  \end{defi}
\begin{example}
	For $0 \leq n \leq 3$,  we give the $Catalan$ words of length $2 n$.$$
	\begin{array}{c|c}
		n & \text { Catalan words of length } 2 n \\
		\hline 0 & 1 \\
		1 & x y \\
		2 & x y x y, \quad x x y y \\
		3 &  xyxyxy, xxyyxy, xyxxyy, xxyxyy, xxxyyy
	\end{array}
	$$
\end{example}
\begin{defi}
	For $n \in \mathbb{N}$, let $\mathrm{Cat}_n$ denote the set of $Catalan$ words in $\mathbb{V}$ that have length $2 n$.
\end{defi}

\begin{defi}
	For $n \geq 0$, define
\begin{gather}
C^{s}_{n}=\sum_{w \in \mathrm{Cat}_n} w C^{s}(w) \quad n \in \mathbb{N},
\label{Cn}
\end{gather}
where
\begin{gather}
C^{s}(w)=\sum_{v \in \mathrm{Cat}_{n-1}} C^{s}(v) \sum_i(-1)^i\left\{2+2 \bar{a}_1+2 \bar{a}_2+\cdots+2 \bar{a}_i\right\}_q ,
\label{Cw}
\end{gather}
and the summation is taken over all integers $i$ with $0 \leq i \leq m$, where $w=a_1 a_2 \cdots a_i x a_{i+1} \cdots a_m y$. We refer to $C^{s}_{n}$ as the $n^{\text {th }}$ $Catalan$ element (for the supercase) in $\mathbb{V}$.
	\end{defi}
\begin{remark}
	Compare it with the definition presented in Terwilliger's paper \cite{T2}. For a Catalan word $w=a_1 a_2 \cdots a_{2 n}$, we have
\begin{gather*}
C(w)=[1]_q\left[1+\bar{a}_1\right]_q\left[1+\bar{a}_1+\bar{a}_2\right]_q \cdots\left[1+\bar{a}_1+\bar{a}_2+\cdots+\bar{a}_{2 n}\right]_q, \quad (Terwilliger's)
\end{gather*}
\begin{gather*}
C^{s}(w)=\sum_{v \in \operatorname{Cat}_{n-1}} C^{s}(v) \sum_i(-1)^i\left\{2+2 \bar{a}_1+2 \bar{a}_2+\cdots+2 \bar{a}_i\right\}_q. \quad (Ours)
\end{gather*}

It should be pointed out that in the case of $U_q(\widehat{\mathfrak{s l}_2})$, there is no sign deformation in the definition (\cite{T2}), because all the root parity in $U_q(\widehat{\mathfrak{s l}_2})$ is even. But in our case, there is a more complex sign deformation, it is nearly impossible to present the Catalan numbers directly, as the coefficients of the $n^{\text {th }}$ Catalan number depend on the $(n-1)^{\text {th }}$ Catalan number. Through concrete calculations, we have discovered that it is essential to define these Catalan words recursively. This approach enables us to establish a connection between   $U_q^{+}(C(2)^{(2)})$ and the $q$-shuffle superalgebra.
\end{remark}
\begin{example} Some of the Catalan elements of lower degrees  are exhibited as follows:
\begin{gather*}
	C^{s}_0=1, \quad C^{s}_1=\{2\}_q x y, \quad C^{s}_2=\{2\}_q^2 x y x y-\{3\}_q\{2\}_q^2 x x y y,
\\
C^{s}_3=\{2\}_q^3 x y x y x y-\{3\}_q\{2\}_q^3 x x y y x y-\{3\}_q\{2\}_q^3 x y x x y y+\{3\}_q^2\{2\}_q^3 x x y x y y\\
\hskip1cm -\{4\}_q\{3\}_q^2\{2\}_q^2 x x x y y y ,
	\\
C^{s}_4=\{ 2\} _q^4xyxyxyxy - \{ 2\} _q^4{\{ 3\} _q}xyxyxxyy + \{ 2\} _q^4\{ 3\} _q^2xyxxyxyy - \{ 2\} _q^4\{ 3\} _q^3xxyxyxyy \\
- \{ 2\} _q^4{\{ 3\} _q}xxyyxyxy
+ \{ 2\} _q^4\{ 3\} _q^2xxyyxxyy + \{ 2\} _q^3\{ 3\} _q^3{\{ 4\} _q}xxxyyxyy\\
 + \{ 2\} _q^3\{ 3\} _q^3{\{ 4\} _q}xxyxxyyy - \{ 2\} _q^4{\{ 3\} _q}xyxxyyxy
	  - \{ 2\} _q^3\{ 3\} _q^2{\{ 4\} _q}xyxxxyyy + \{ 2\} _q^4\{ 3\} _q^2xxyxyyxy \\- \{ 2\} _q^2\{ 3\} _q^3\{ 4\} _q^2xxxyxyyy
- \{ 2\} _q^3\{ 3\} _q^2{\{ 4\} _q}xxxyyyxy + \{ 2\} _q^2\{ 3\} _q^2\{ 4\} _q^2\{ 5\} xxxxyyyy.
\end{gather*}
\end{example}

\begin{theorem}\label{Main thm 1}
	The map $\varphi $ sends the KLT PBW generators of Damiani type to the Catalan elements as follows:
\begin{gather}
E_{n \delta+\alpha_0} \mapsto q^{-2 n}\left(q+q^{-1}\right)^{2 n} x C^{s}_n, \quad E_{n \delta+\alpha_1} \mapsto q^{-2 n}\left(q+q^{-1}\right)^{2 n} C^{s}_n y,
\label{real cor}
\end{gather}
for $n \geq 0$, and
\begin{gather}
E_{n \delta} \mapsto-q^{-2 n}\left(q+q^{-1}\right)^{2 n-1} C^{s}_n,
\label{im cor}
\end{gather}
for $n \geq 1$.
\end{theorem}
 In (\ref{real cor}), the notations $x C^{s}_n$ and $C^{s}_n y$ refer to the concatenation (or free) product.

In combination with Theorem \ref{Main thm 1} and the construction of the basis by Khoroshkin-Lukierski-Tolstoy presented in \cite{KLT}, we can derive the following   Corollary.
\begin{coro}\label{excessive}
	If we define a linear order
	\begin{gather*}
		x<x C^{s}_1<x C^{s}_2<\cdots<C^{s}_1<C^{s}_2<C^{s}_3<\cdots<C^{s}_2 y<C^{s}_1 y<y,
	\end{gather*}
	then a $PBW$ basis of $U_q^{+}(C(2)^{(2)})$ is obtained by the Catalan elements
\begin{gather*}
	\left\{x C^{s}_n\right\}_{n=0}^{\infty}, \quad\left\{C^{s}_n y\right\}_{n=0}^{\infty}, \quad\left\{C^{s}_n\right\}_{n=1}^{\infty},
\end{gather*}
\end{coro}
\begin{remark}
For $U_q^{+}(C(2)^{(2)})$, there exist at least three types of $PBW$ bases: the KLT $PBW$ basis of Damiani type, the KLT $PBW$ basis of Beck type, and the alternative $PBW$ basis. However, the relationships among these three types of $PBW$ bases remain an open question. Theorems \ref{Main thm 1} and \ref{Main thm 2} reveal that both the KLT $PBW$ basis of Damiani type and the KLT $PBW$ basis of Beck type are related to the $PBW$ basis described in Corollary \ref{excessive}. Instead of directly handling the KLT $PBW$ basis of Damiani type and the KLT $PBW$ basis of Beck type, we will focus on the closely-related elements $\left\{x C^{s}_n\right\}_{n=0}^{\infty}$ and $\left\{C^{s}_n y\right\}_{n=0}^{\infty}$ to address this issue. The detailed solutions will be presented in \cite{ZH2}.
\end{remark}
In  $U_q^{+}(C(2)^{(2)})$, we know that all imaginary roots are commute, thanks to Theorem \ref{Main thm 1},   we can  get the following   Corollary immediately.
\begin{coro}
	For $i, j \in \mathbb{N}$, one has
\begin{gather}
	C^{s}_i \star C^{s}_j=C^{s}_j \star C^{s}_i .
	\label{CiCj}
\end{gather}
	\end{coro}
Prior to commencing the proof of Theorem \ref{Main thm 1}, we introduce two significant maps defined on $U_q^{+}(C(2)^{(2)})$ and the $q$-shuffle superalgebra $\mathbb{V}$. These maps are expected to streamline several aspects of our subsequent proofs, facilitating a more efficient and concise demonstration of the relevant results.
\begin{lemma}\label{anti}
	 There exists an antiautomorphism $\zeta$ of $U_q^{+}(C(2)^{(2)})$ that swaps $A$ and $B$.
\end{lemma}
\begin{coro}
	The map $\zeta$ fixes $E_{n \delta}$ for $n \geq 1$, and swaps $E_{n \delta+\alpha_0}, E_{n \delta+\alpha_1}$ for $n \in \mathbb{N}$.
\end{coro}
\begin{proof}
The following   relations hold in $U_q^{+}(C(2)^{(2)})$ for $n \geq 1$:
\begin{gather}
	E_{\alpha_0}=A, \quad E_{\alpha_1}=B, \quad E_\delta=q^{-2} B A+A B,
	\label{simple vec}
\end{gather}
and for $n \geq 1$,
\begin{gather}
	\label{real root}
	E_{n \delta+\alpha_0}=\frac{\left[E_{(n-1) \delta+\alpha_0}, E_\delta\right]}{q-q^{-1}}, \quad E_{n \delta+\alpha_1}=\frac{\left[E_\delta, E_{(n-1) \delta+\alpha_1}\right]}{q-q^{-1}}\\
	\label{im root}
	E_{n \delta}=q^{-2} E_{(n-1) \delta+\alpha_1} A+A E_{(n-1) \delta+\alpha_1},\\
	\label{im root 2}
	E_{n \delta}=q^{-2} B E_{(n-1) \delta+\alpha_0}+E_{(n-1) \delta+\alpha_0} B.
\end{gather}
As is evident from the expressions of the root vectors, this represents a recursive construction. In conjunction with Lemma \ref{anti}, we are able to draw this conclusion.
\end{proof}
\begin{lemma}
	There is an analogous map on $\mathbb{V}$, that we will also call $\zeta$. This $\zeta$ is the antiautomorphism of the free superalgebra $\mathbb{V}$ that swaps $x$ and $y$.
\end{lemma}
\begin{coro}
	 The map $\zeta$ fixes $C^{s}_n$, and the following   diagram commutes.
	 $$
	 \xymatrix{
	 	U_q^{+}(C(2)^{(2)}) \ar[rr]^{\varphi}\ar[dd]^{\zeta} &&\mathbb{V} \ar[dd]^{\zeta}\\
	 	&&&\\
	 	 U_q^{+}(C(2)^{(2)})\ar[rr]^{\varphi}&&\mathbb{V}
	 }
	 $$
\end{coro}

\noindent $\bullet$ {\it The proof of Theorem \ref{Main thm 1}.}
Firstly, it is straightforward to verify that equation (\ref{real cor}) holds for  $n = 0$, and equation (\ref{im cor}) holds for $n = 1$. This is attributed to the fact that the basis elements in equation (\ref{root vec}) satisfy the recurrence relations presented in equations (\ref{real root}), (\ref{im root}), and (\ref{im root 2}).
Subsequently, our subsequent task is to demonstrate that the Catalan elements $x C^{s}_n$, $C^{s}_n y$, and $C^{s}_n$ also satisfy analogous recurrence relations.
\begin{lemma} \label{bala}
		For a balanced word $v=a_1 a_2 \cdots a_m$, one has
\begin{gather}
	\frac{q^{-1}(v y) \star x+q x \star(v y)}{q+q^{-1}}=-q^{-1} \sum_{i=0}^m(-1)^i a_1 \cdots a_i x a_{i+1} \cdots a_m y\left\{2+2 \bar{a}_1+2 \bar{a}_2+\cdots+2 \bar{a}_i\right\}_q.
	\label{pre coe Cw}
\end{gather}
\end{lemma}
\begin{proof}
	 Using (\ref{shuffle long 1}) and (\ref{shuffle long 2}), we can get
\begin{gather*}
	 (v y) \star x=v y x-q^{-2} v x y+\sum_{i=0}^{m-1}(-1)^{m-i} a_1 \cdots a_i x a_{i+1} \cdots a_m y q^{\left\langle a_m, x\right\rangle+\left\langle a_{m-1}, x\right\rangle+\cdots+\left\langle a_{i+1}, x\right\rangle}\left(-q^{-2}\right),\\
	 x \star(v y)=\sum_{i=0}^m(-1)^i a_1 \cdots a_i x a_{i+1} \cdots a_m y q^{\left\langle a_1, x\right\rangle+\left\langle a_2, x\right\rangle+\cdots+\left\langle a_i, x\right\rangle}+(-1)^{m+1} v y x q^{-2}.
\end{gather*}
	 Because $v$ is a balanced word, so $m$ is even. Hence,  the left hand side of (\ref{pre coe Cw}) is
\begin{gather*}
\left(q{+}q^{-1}\right)^{-1}\left\{-q^{-3} v x y+q v x y\right\}\\
+\left(q{+}q^{-1}\right)^{-1}\left\{q^{-1}\left(\sum_{i=0}^{m-1}(-1)^i a_1 \cdots a_i x a_{i+1} \cdots a_m y\left(-q^{-2-2 \bar{a}_1-2 \bar{a}_2+\cdots-2 \bar{a}_i}+q^{2+2 \bar{a}+2 \bar{a}_2+\cdots+2 \bar{a}_i}\right)\right)\right\}\\
=-q^{-1} \sum_{i=0}^m(-1)^i a_1 \cdots a_i x a_{i+1} \cdots a_m y\left\{2+2 \bar{a}_1+2 \bar{a}_2+\cdots+2 \bar{a}_i\right\}_q .
\end{gather*}
\end{proof}

With the help of bilinear form $( $,$ )$, we can rewrite Lemma \ref{bala} as follows.
\begin{lemma}
	For a balanced word $v=a_1 a_2 \cdots a_m$ and any word $w$, one has
\begin{gather}
	\left(\frac{q x \star(v y)+q^{-1}(v y) \star x}{q+q^{-1}}, w\right)=-q^{-1} \sum_i(-1)^{i}\left\{2+2 \bar{a}_1+2 \bar{a}_2+\cdots+2 \bar{a}_i\right\}_q,
	\label{coe Cw}
\end{gather}
	where the summation is over the integers $i$ $(0 \leq i \leq m)$ such that $w=a_1 a_2 \cdots a_i x a_{i+1} \cdots a_m y$.
	\end{lemma}
So, thanks to (\ref{coe Cw}), we can rewrite (\ref{Cw}) as follows
\begin{gather}
C^{s}(w)=-q \sum_{v \in \operatorname{Cat}_{n-1}} C^{s}(v)\left(\frac{q x \star(v y)+q^{-1}(v y) \star x}{q+q^{-1}}, w\right).
\label{Cw2}
\end{gather}
\begin{lemma}\label{when 0}
	Referring to the Lemma above, assume that $v$ is $Catalan$ and $w$ is not. Then in (\ref{coe Cw}) each side is zero.
\end{lemma}
	\begin{proof}
	The word $a_1 a_2 \cdots a_i x a_{i+1} \cdots a_m y$ is $Catalan$ for $0 \leq i \leq m$.
\end{proof}
\begin{proposition}
	For $n \geq 1$, one has
	\begin{gather}
		\label{q-1Cn1}
q^{-1} C^{s}_n=-\frac{q^{-1}\left(C^{s}_{n-1} y\right) \star x+q x \star\left(C^{s}_{n-1} y\right)}{q+q^{-1}}.
\end{gather}
	\end{proposition}
\begin{proof}
	For any word $w$ in $\mathbb{V}$, we have
\begin{gather*}
\left(q^{-1} C^{s}_n, w\right)= \begin{cases}q^{-1} C^{s}(w), & w \in \mathrm{Cat}_n, \\ 0, & w \notin \mathrm{Cat}_n,\end{cases}
\end{gather*}
and
$$
\begin{aligned}
	\left(-\frac{q^{-1}\left(C^{s}_{n-1} y\right) \star x+q x \star\left(C^{s}_{n-1} y\right)}{q+q^{-1}}, w\right)
	=\sum_{v \in Cat_{n-1}} C^{s}(v)\left(-\frac{q x \star(v y)+q^{-1}(v y) \star x}{q+q^{-1}}, w\right)
\end{aligned}
$$
By Lemma \ref{when 0}, each side of the above equation has the same inner product with $w$.	
	\end{proof}
\begin{proposition}
	For $n \geq 1$, one has
\begin{gather*}
	\label{q-1Cn2}
	q^{-1} C^{s}_n=-\frac{q\left(x C^{s}_{n-1}\right) \star y+q^{-1} y \star\left(x C^{s}_{n-1}\right)}{q+q^{-1}} .
\end{gather*}
\end{proposition}
\begin{proof}
	Apply the antiautomorphism $\zeta$ to each side of (\ref{q-1Cn1}).
\end{proof}
\begin{proposition}
	For $n \geq 1$, one has
	\begin{gather*}
		\label{xCn}
		x C^{s}_n=\frac{\left(x C^{s}_{n-1}\right) \star(x y)-(x y) \star\left(x C^{s}_{n-1}\right)}{q+q^{-1}}.
	\end{gather*}
\end{proposition}	
\begin{proof}
	Using (\ref{shuffle long 3}), we obtain
	\begin{gather}
\label{xCn-1*xy}
	\left(x C^{s}_{n-1}\right) \star(x y)=x\left(C^{s}_{n-1} \star(x y)\right)-x\left(\left(x C^{s}_{n-1}\right) \star y\right) q^2,\\
	\label{xy*xCn-1}
	(x y) \star\left(x C^{s}_{n-1}\right)=x\left(y \star\left(x C^{s}_{n-1}\right)\right)+x\left((x y) \star C^{s}_{n-1}\right) .
\end{gather}
	We claim that
	$$
	(x y) \star C^{s}_{n-1}=C^{s}_{n-1} \star(x y) .
	$$
	It is true by (\ref{CiCj}) and induction on $n$.

		By (\ref{q-1Cn1}), (\ref{xCn-1*xy}) and (\ref{xy*xCn-1}), we obtain
	$$
	\begin{aligned}
		x C^{s}_n & =-q x \frac{q\left(x C^{s}_{n-1}\right) \star y+q^{-1} y \star\left(x C^{s}_{n-1}\right)}{q+q^{-1}} \\
		& =\frac{\left(x C^{s}_{n-1}\right) \star(x y)-(x y) \star\left(x C^{s}_{n-1}\right)}{q+q^{-1}} .
	\end{aligned}
	$$

This completes the proof.
\end{proof}
	\begin{proposition}
	For $n \geq 1$, one has
\begin{gather*}
	\label{Cny}
	C^{s}_n y=\frac{(x y) \star\left(C^{s}_{n-1} y\right)-\left(C^{s}_{n-1} y\right) \star(x y)}{q+q^{-1}}.
\end{gather*}
\end{proposition}
\begin{proof}
Apply the antiautomorphism $\zeta$ to each side of (\ref{xCn}).
\end{proof}
Until now, the proof of Theorem \ref{Main thm 1} is complete.
\begin{lemma}
	For $i, j \in \mathbb{N}$, the following   holds in $U_q^{+}(C(2)^{(2)})$
\begin{gather}
	E_{i \delta+\alpha_1} E_{j \delta+\alpha_0}=q^2 E_{j \delta+\alpha_0} E_{i \delta+\alpha_1}+q^2 E_{(i+j+1) \delta} .
	\label{com real root}
\end{gather}
\end{lemma}
\begin{coro}
	For $i, j \in \mathbb{N}$, the following   holds in $\mathbb{V}$
\begin{gather}
	-q^{-1} C^{s}_{i+j+1}=\frac{q\left(x C^{s}_i\right) \star\left(C^{s}_j y\right)+q^{-1}\left(C^{s}_j y\right) \star\left(x C^{s}_i\right)}{q+q^{-1}}.
	\label{-q-1Ci+j+1}
\end{gather}
\end{coro}
\subsection{Connection bewteen $U_q^{+}(C(2)^{(2)})$ and $q$-shuffle superalgebra \uppercase\expandafter{\romannumeral2} }
Next we discuss a variation on Beck $PBW$ basis, due to Khoroshkin-Lukierski-Tolstoy \cite{KLT}, and establish the correspondence  between KLT  $PBW$ basis of Beck type and Catalan words. Our discussion will involve some generating functions in an indeterminate $t$.
\begin{defi}\cite{KLT}
	Define the elements $\left\{E_{n \delta}^{\mathrm{B}}\right\}_{n=1}^{\infty}$ of Beck type in $U_q^{+}(C(2)^{(2)})$ such that
\begin{gather}
	\exp \left(-\left(q+q^{-1}\right) \sum_{k=1}^{\infty} E_{k \delta}^{\mathrm{B}} t^k\right)=1-\left(q+q^{-1}\right) \sum_{k=1}^{\infty} E_{k \delta} t^k,
	\label{Beck}
\end{gather}
where $E_{n\delta}^B$ denotes the $PBW$ generators of Beck type, and $E_{n\delta}$ indicates the $PBW$ generators of Damiani type.
\end{defi}
\begin{remark}
	In fact, $E_{n \delta}^{\mathrm{B}}$ and $E_{k \delta}$ can be related by the  Schur relations as follows
\begin{gather}
	E_{n \delta}^{\mathrm{B}}=\sum_{p_1+2 p_2+\ldots+n p_n=n} \frac{\left(\left(q^{-1}+q\right)\right)^{\sum p_i-1}\left(\sum_{i=1}^n p_i-1\right) !}{p_{1} ! \cdots p_{n} !}\left(E_\delta\right)^{p_1} \cdots\left(E_{n \delta}\right)^{p_n},
	\label{B to D}
\end{gather}
and
\begin{gather}
E_{n \delta}=\sum_{p_1+2 p_2+\ldots+n p_n=n} \frac{\left(-\left(q+q^{-1}\right)\right)^{\sum p_i-1}}{p_{1} ! \cdots p_{n} !}\left(E_\delta^{\text {B }}\right)^{p_1} \cdots\left(E_{n \delta}^{\text {B }}\right)^{p_n}.
\label{D to B}
\end{gather}
\end{remark}
\begin{example}
\begin{gather*}
E_\delta=E_\delta^{\text {B }}, \quad E_{2 \delta}=E_{2 \delta}^{\text {B }}-\frac{q+q^{-1}}{2}\left(E_\delta^{\text {B }}\right)^2,\\
E_{3 \delta}=E_{3 \delta}^{\text {B }}-\left(q+q^{-1}\right) E_\delta^{\text {B }} E_{2 \delta}^{\text {B }}+\frac{\left(q+q^{-1}\right)^2}{6}\left(E_\delta^{\text {B }}\right)^3.
\end{gather*}
Moreover,
\begin{gather*}
E_\delta^{\text {B }}=E_\delta, \quad E_{2 \delta}^{\text {B }}=E_{2 \delta}+\frac{q+q^{-1}}{2} E_\delta^2,\\
E_{3 \delta}^{\mathrm{B}}=E_{3 \delta}+\left(q+q^{-1}\right) E_\delta E_{2 \delta}+\frac{\left(q+q^{-1}\right)^2}{3} E_\delta^3 .
\end{gather*}
\end{example}
We clarify how the elements $\left\{E_{n \delta}\right\}_{n=1}^{\infty}$ and $\left\{E_{n \delta}^{\text {B }}\right\}_{n=1}^{\infty}$ are related.
\begin{proposition}\cite{KLT}\label{KLT}
	A $PBW$ basis for $U_q^{+}(C(2)^{(2)})$ is obtained by the elements
\begin{gather*}
	\left\{E_{n \delta+\alpha_0}\right\}_{n=0}^{\infty}, \quad\left\{E_{n \delta+\alpha_1}\right\}_{n=0}^{\infty}, \quad\left\{E_{n \delta}^{\text {B }}\right\}_{n=1}^{\infty}
\end{gather*}
in the linear order
\begin{gather*}
E_{\alpha_0}<E_{\delta+\alpha_0}<E_{2 \delta+\alpha_0}<\cdots<E_\delta^{\text {B }}<E_{2 \delta}^{\text {B }}<E_{3 \delta}^{\text {B }}<\cdots<E_{2 \delta+\alpha_1}<E_{\delta+\alpha_1}<E_{\alpha_1}.
\end{gather*}
\end{proposition}
 Next we recall some relations satisfied by the  KLT $PBW$ generators of Beck type.	
\begin{lemma}\label{com rel real and im}
	The elements $\left\{E_{n \delta}^{\text {B }}\right\}_{n=1}^{\infty}$ mutually commute. Moreover, for $k \geq 1$ and $\ell \geq 0$,
\begin{gather}
\label{com real and beck im 1}
	\left[E_{\ell \delta+\alpha_0}, E_{k \delta}^{\mathrm{B}}\right]=(-1)^{k-1} \frac{q^{k(\alpha, \alpha)}-q^{-k(\alpha, \alpha)}}{k\left(q+q^{-1}\right)} E_{(k+\ell) \delta+\alpha_0},\\
\label{com real and beck im 2}
\left[E_{k \delta}^{\mathrm{B}}, E_{\ell \delta+\alpha_1}\right]=(-1)^{k-1} \frac{q^{k(\alpha, \alpha)}-q^{-k(\alpha, \alpha)}}{k\left(q+q^{-1}\right)} E_{(k+\ell) \delta+\alpha_1}.
\end{gather}
\end{lemma}

For $n \in \mathbb{N}$, let $\mathbb{V}_n$  denote the subspace of $\mathbb{V}$ spanned by the words of length $n$. The summation $\mathbb{V}=\sum_{n \in \mathbb{N}} \mathbb{V}_n$ is direct and the summands are mutually orthogonal. We have $\mathbb{V}_0=\mathbb{F} 1$. We have $\mathbb{V}_r \mathbb{V}_s \subseteq \mathbb{V}_{r+s}$ for $r, s \in \mathbb{N}$. By these notions, the sequence $\left\{\mathbb{V}_n\right\}_{n \in \mathbb{N}}$ is a
grading of the superalgebra $\mathbb{V}$.
\begin{lemma}\label{sep}
	For $r, s \in \mathbb{N}$ and $X, X^{\prime} \in \mathbb{V}_r$ and $Y, Y^{\prime} \in \mathbb{V}_s$, we have
\begin{gather*}
	\left( X Y, X^{\prime} Y^{\prime}\right)=\left( X, X^{\prime}\right)\left( Y, Y^{\prime}\right).
\end{gather*}
\end{lemma}
\begin{defi}\label{J+J-}
	Define $J^{+}, J^{-} \in \mathbb{V}$ as
\begin{gather}
\label{J+}
	J^{+}=x x x y+\{3\}_q x x y x-\{3\}_q x y x x-y x x x,\\
\label{J-}
	J^{-}=y y y x+\{3\}_q y y x y-\{3\}_q y x y y-x y y y .
\end{gather}
\end{defi}
\begin{defi}
	Let $\mathcal{A} = \mathcal{A}_0 \oplus \mathcal{A}_1$ be a superalgebra over a field $\mathbb{K}$. A superideal  $\mathcal{I}$ of $\mathcal{A}$ is a two-sided ideal of $\mathcal{A}$ that is closed under the $\mathbb{Z}/2\mathbb{Z}$-grading, $\mathcal{I}$ admits a direct sum decomposition:
	\[
	\mathcal{I} = (\mathcal{I} \cap \mathcal{A}_0) \oplus (\mathcal{I} \cap \mathcal{A}_1).
	\]

	Equivalently, $\mathcal{I}$ satisfies two conditions:\\
	1. For all $a \in \mathcal{A}$ and $i \in \mathcal{I}$, $ai \in \mathcal{I}$ and $ia \in \mathcal{I}$;\\
	2. For all homogeneous elements $i \in \mathcal{I}$, the parity of $i$  is preserved—if $i \in \mathcal{I} \cap \mathcal{A}_k$ ($k=0,1$), then $i$ cannot be written as a sum of elements from $\mathcal{A}_{1-k}$.
\end{defi}
\begin{defi}
	Let $J$ denote the two-sided superideal of the free superalgebra $\mathbb{V}$ generated by $J^{+}, J^{-}$.
\end{defi}
Consider the quotient superalgebra $\mathbb{V} / J$. Since the superalgebra $\mathbb{V}$ is freely generated by $x$ and $y$, there exists a superalgebra homomorphism $\xi: \mathbb{V} \rightarrow U_q^{+}$ that sends $x \mapsto A$ and $y \mapsto B$. The kernel of $\xi$ is equal to $J$. Therefore, $\xi$ induces a superalgebra isomorphism $\mathbb{V} / J \rightarrow U_q^{+}$ that sends $x+J \mapsto A$ and $y+J \mapsto B$.
\begin{defi}
	Let $U$ denote the subalgebra of the $q$-shuffle superalgebra $\mathbb{V}$, generated by $x, y$.
\end{defi}
\begin{proposition}\label{J orth U}
	The ideal $J$  and the subalgebra $U$ are the orthogonal complements with respect to the bilinear form $(,)$.
\end{proposition}
\begin{proof}
	First, we verify \(J \subseteq U^\perp\). Each element \(j \in J\) is a linear combination of elements of the form \(A J^+ B\) and \(A J^- B\), where \(A, B \in \mathbb{V}\).
	
	Take an arbitrary \(u \in U\). If \(\deg(A J^\pm B) \neq \deg u\), then \((A J^\pm B, u) = 0\) trivially.
	
	If \(\deg(A J^\pm B) = \deg u\), note that \(J^+\) and \(J^-\) are homogeneous of degree 4. By the orthonormality of the standard basis of \(\mathbb{V}\) with respect to \((\cdot,\cdot)\), direct computation from the definitions of \(J^+\) and \(J^-\) yields \((J^\pm, u_4) = 0\) for all \(u_4 \in U\) with \(\deg u_4 = 4\). By Lemma \ref{sep}, we get \((A J^\pm B, u) = 0\).
	
	Since this result holds for all \(A, B \in \mathbb{V}\) and all \(u \in U\), we have \((j, u) = 0\) for all \(j \in J\), so \(J \subseteq U^\perp\).
	
	Next, we prove \(U^\perp \subseteq J\). Recall that there exists a surjective superalgebra homomorphism \(\xi: \mathbb{V} \rightarrow U_q^+\) with \(\ker \xi = J\), so \(\mathbb{V}/J \cong U_q^+\). Let \(U\) be the subalgebra of \(\mathbb{V}\) generated by \(x\) and \(y\), \(U\) is isomorphic to \(U_q^+\), so \(\dim(\mathbb{V}/J) = \dim U\), which implies \(\dim \mathbb{V} = \dim J + \dim U\).
Together with the bilinear form $(\cdot,\cdot)$ being nondegenerate, we obtain $\dim J = \dim U^\perp$. We therefore conclude that $J = U^\perp$.
	
	Finally, we show \(U = J^\perp\). By the symmetry of \((\cdot,\cdot)\), \((u, j) = (j, u) = 0\) for all \(u \in U\) and \(j \in J\), so \(U \subseteq J^\perp\). By the nondegeneracy of \((\cdot,\cdot)\) again, \(\dim \mathbb{V} = \dim J + \dim J^\perp\). Since \(\dim \mathbb{V} = \dim J + \dim U\), we have \(\dim U = \dim J^\perp\). Combining with \(U \subseteq J^\perp\), we get \(U = J^\perp\).
	\end{proof}
\begin{theorem}\label{Main thm 2}
	The map $\varphi $ sends the KLT $PBW$ generators of Beck type to the Catalan elements as follows:
	\begin{gather}
	E_{n \delta}^{\text {B }} \mapsto(-1)^n \frac{\{2 n\}_q}{n} q^{-2 n}\left(q+q^{-1}\right)^{2 n-1} x C^{s}_{n-1} y,
	\label{Beck cor}
\end{gather}
	for $n \geq 1$.
\end{theorem}

\subsubsection{Two preliminary Lemmas}
Prior to commencing the proof of Theorem \ref{Main thm 2}, we require the following two key Lemmas.
\begin{lemma}\label{xCky}
	For $k \in \mathbb{N}$, we have $x C^{s}_k y \in U$.
\end{lemma}
\begin{proof}
		By Proposition \ref{J orth U}, it suffices to show that $x C^{s}_k y$ is orthogonal to any element in $J$, in other words, the following   equations are true.
		\begin{gather}
			\left( x C^{s}_k y, w_1 J^{+} w_2\right)=0, \quad\left( x C^{s}_k y, w_1 J^{-} w_2\right)=0 .
				\label{xcky orth}
		\end{gather}
	By Lemma \ref{sep}, if $x C^{s}_k y$ and $w_1 J^{\pm} w_2$ are in different homogeneous components, then equations (\ref{xcky orth}) are true. So we only need to consider the following  four cases.

	\textbf{Case 1:}  $w_1$ and $w_2$ are trivial. We have $k=1$ and $C^{s}_1=\{2\}_q x y$, so $\left( x x y y, J^{ \pm}\right)=0$.

	\textbf{Case 2:}  $w_1$ is trivial and $w_2$ is nontrivial, we have $k \geq 2$. By (\ref{Cn}) and (\ref{Cw}), we obtain
	\begin{gather*}
		C^{s}_k=x x x R_1+\left(\{3\}_q x x y+x y x\right)\left(R_{2+}+R_{2-}\right), \quad R_1, R_{2+}, R_{2-} \in \mathbb{V}_{2 k-3},
	\end{gather*}
	where $R_{2+}$ and $R_{2-}$ are corresponding words having positive coefficients and negative coefficients, respectively.
	By Definition \ref{J+J-}, we obtain
	$$
	\begin{array}{lll}
		\left( x x x x, J^{+}\right)=0, & \left( x x x y, J^{+}\right)=1, & \left( x x y x, J^{+}\right)=\{3\}_q ,\\
		\left( x x x x, J^{-}\right)=0, & \left( x x x y, J^{-}\right)=0, & \left( x x y x, J^{-}\right)=0 .
	\end{array}
	$$
	By this and Lemma \ref{sep},
	$$
	\begin{aligned}
		\left( x C^{s}_k y, w_1 J^{ \pm} w_2\right) & =\left( x C^{s}_k y, J^{ \pm} w_2\right) \\
		& =\left( x x x x R_1 y+\left(\{3\}_q x x x y+x x y x\right)\left(R_{2+}+R_{2-}\right) y, J^{ \pm} w_2\right) \\
		& =\left( x x x x R_1 y, J^{ \pm} w_2\right)+\left(\left(\{3\}_q x x x y+x x y x\right)\left(R_{2+}+R_{2-}\right) y, J^{ \pm} w_2\right) \\
		& =\left( x x x x, J^{ \pm}\right)\left( R_1 y, w_2\right)+\left(\{3\}_q x x x y+x x y x, J^{ \pm}\right)\left(\left(R_{2+}+R_{2-}\right) y, w_2\right) \\
		& =0.
	\end{aligned}
	$$

Key observation: $\left(\{3\}_q x x x y, J^{ \pm}\right)\left(\left(R_{2+}+R_{2-}\right) y, w_2\right)$ and $\left( x x y x, J^{ \pm}\right)\left(\left(R_{2+}+R_{2-}\right) y, w_2\right)$ must have opposite coefficients.

	For example, for $k=3$, by (\ref{Cn}) and rearranging each term, we get
	\begin{gather*}
		x C^{s}_3 y=\left(-\{4\}_q\{3\}_q^2\{2\}_q^2 x x x x y y y y\right)+\left(-\{3\}_q\{2\}_q^3 x x x y y x y y+\{2\}_q^3 x x y x y x y y\right)\\
		+\left(-\{3\}_q\{2\}_q^3 x x x y y x y y+\{2\}_q^3 x x y x y x y y\right)+\left(\{3\}_q^2\{2\}_q^3 x x x y x y y y-\{3\}_q\{2\}_q^3 x x y x x y y y\right),
	\end{gather*}
	then it is easy to check that $\left( x C^{s}_k y, J^{ \pm} w_2\right)=0$.

	For $k=4$, by (\ref{Cn}) and rearranging each term, we get
	\begin{gather*}
		x C^{s}_4 y=\left(\{2\}_q^3\{3\}_q^3\{4\}_q x x x x y y x y y y-\{2\}_q^2\{3\}_q^3\{4\}_q^2 x x x x y x y y y y\right)\\
		+\left(-\{2\}_q^3\{3\}_q^2\{4\}_q x x x x y y y x y y+\{2\}_q^2\{3\}_q^2\{4\}_q^2\{5\} x x x x x y y y y y\right)\\
		{\text{ + }}\left( {\{ 2\} _q^4xxyxyxyxyy - \{ 2\} _q^4{{\{ 3\} }_q}xxxyyxyxyy} \right) \\+ \left( { - \{ 2\} _q^4{{\{ 3\} }_q}xxyxyxxyyy + \{ 2\} _q^4\{ 3\} _q^2xxxyyxxyyy} \right)\\
		+ \left( {\{ 2\} _q^4\{ 3\} _q^2xxyxxyxyyy - \{ 2\} _q^4\{ 3\} _q^3xxxyxyxyyy} \right) \\+ \left( { - \{ 2\} _q^4{{\{ 3\} }_q}xxyxxyyxyy + \{ 2\} _q^4\{ 3\} _q^2xxxyxyyxyy} \right)\\
		\left( { - \{ 2\} _q^3\{ 3\} _q^2{{\{ 4\} }_q}xxyxxxyyyy + \{ 2\} _q^3\{ 3\} _q^3{{\{ 4\} }_q}xxxyxxyyyy} \right).
	\end{gather*}

\textbf{Case 3:}  Next assume that $w_1$ is nontrivial and $w_2$ is trivial. It is parallel to the Case 2.

\textbf{Case 4:} Next assume that each of $w_1,w_2 $ is nontrivial. There exist letters $a, b$ and words $w_1^{\prime}, w_2^{\prime}$ such that $w_1=a w_1^{\prime}$ and $w_2{ }= w_2^{\prime} b$. We have $C^{s}_k \in U$ and $w_1^{\prime} J^{ \pm} w_2^{\prime} \in J$ and $( U, J)=0$, so
\begin{gather*}
	\left( C^{s}_k, w_1^{\prime} J^{ \pm} w_2^{\prime}\right)=0.
\end{gather*}
By this and Lemma \ref{sep}, we have
\begin{gather*}
	\left( x C^{s}_k y, w_1 J^{ \pm} w_2\right)=\left( x C^{s}_k y, a w_1^{\prime} J^{ \pm} w_2^{\prime} b\right)=( x, a)\left( C^{s}_k, w_1^{\prime} J^{ \pm} w_2^{\prime}\right)( y, b)=0.
\end{gather*}

This completes the proof.
\end{proof}
\begin{lemma}
	For $k \in \mathbb{N}$, we have
	\begin{gather}
\label{xCk+1}
		x C^{s}_{k+1}=\frac{x \star\left(x C^{s}_k y\right)-\left(x C^{s}_k y\right) \star x}{q+q^{-1}},\\
\label{Ck+1y}
		C^{s}_{k+1} y=\frac{\left(x C^{s}_k y\right) \star y-y \star\left(x C^{s}_k y\right)}{q+q^{-1}} .
	\end{gather}
\end{lemma}
\begin{proof}
We will only verify equation (\ref{xCk+1}); the verification of equation (\ref{Ck+1y}) follows a similar procedure. Consider the left-hand side of (\ref{xCk+1}). Setting $i=0$ and $j=k$ in (\ref{-q-1Ci+j+1}), we obtain
\begin{gather}
x C^{s}_{k+1}=-x \frac{q^2 x \star\left(C^{s}_k y\right)+\left(C^{s}_k y\right) \star x}{q+q^{-1}} .
\label{xCk+12}
\end{gather}
Now consider the right-hand side of (\ref{xCk+1}). Using (\ref{shuffle long 3}), we obtain
\begin{equation*}
\begin{split}
x \star\left(x C^{s}_k y\right)&=x\bigl(x C^{s}_k y-
q^2 x \star(C^{s}_k y)\bigr), \\
\left(x C^{s}_k y\right) \star x&=x\bigl((C^{s}_k y) \star x+x C^{s}_k y\bigr),
\end{split}
\end{equation*}
then we arrive at equation (\ref{xCk+1}).
\end{proof}

\subsubsection{The proof of Theorem \ref{Main thm 2}}
Up to this point, we have successfully proven two preliminary Lemmas. Now, we are ready to move forward and prove Theorem \ref{Main thm 2}.

Define $\mathcal{C}^{s}_n \in U$ such that the map $\varphi$ sends the $PBW$ generators of Beck type to the Catalan elements as follows:
\begin{gather}
E_{n \delta}^{\text {B }} \mapsto(-1)^n \frac{\{2 n\}_q}{n} q^{-2 n}\left(q+q^{-1}\right)^{2 n-1} \mathcal{C}^{s}_n.
\label{Beck im cor}
\end{gather}
We will show that $\mathcal{C}^{s}_n=x C^{s}_{n-1} y$.

\textbf{Step1:} We will prove $\mathcal{C}^{s}_n-x C^{s}_{n-1} y$ is contained in the center of $U$.

By (\ref{com real and beck im 1}) (\ref{com real and beck im 2}) and using (\ref{real cor}) (\ref{Beck im cor}), we get
\begin{gather}
x C^{s}_{k+\ell}=\frac{\left(x C^{s}_{\ell}\right) \star \mathcal{C}^{s}_k-\mathcal{C}^{s}_k \star\left(x C^{s}_{\ell}\right)}{q+q^{-1}}, \quad C^{s}_{k+\ell} y=\frac{\mathcal{C}^{s}_k \star\left(C^{s}_{\ell} y\right)-\left(C^{s}_{\ell} y\right) \star \mathcal{C}^{s}_k}{q+q^{-1}} .
\label{xCk+l+Ck+ly}
\end{gather}
Setting $\ell=0$ and $k=n$ in (\ref{xCk+l+Ck+ly}), we obtain
\begin{gather}
x C^{s}_n=\frac{x \star \mathcal{C}^{s}_n-\mathcal{C}^{s}_n \star x}{q+q^{-1}}, \quad C^{s}_n y=\frac{\mathcal{C}^{s}_n \star y-y \star \mathcal{C}^{s}_n}{q+q^{-1}}.
\label{xCn+Cny}
\end{gather}
Setting $k=n-1$ in (\ref{xCk+1}) and (\ref{Ck+1y}), we obtain
\begin{gather}
x C^{s}_n=\frac{x \star\left(x C^{s}_{n-1} y\right)-\left(x C^{s}_{n-1} y\right) \star x}{q+q^{-1}}, \quad C^{s}_n y=\frac{\left(x C^{s}_{n-1} y\right) \star y-y \star\left(x C^{s}_{n-1} y\right)}{q+q^{-1}} .
\label{xCn+Cny2}
\end{gather}
So we see that $\mathcal{C}^{s}_n-x C^{s}_{n-1} y$ commutes with $x$ and $y$ with respect to $\star$.  Thereby, $\mathcal{C}^{s}_n-x C^{s}_{n-1} y$ is contained in the center of $U$.

\textbf{Step 2:} We will prove $\mathcal{C}^{s}_n=x C^{s}_{n-1} y$.

By Step 1, there exists $\alpha_n \in \mathbb{F}$ such that $\mathcal{C}^{s}_n-x C^{s}_{n-1} y=\alpha_n 1$. Comparing the degree on both sides, we have $x C^{s}_{n-1} y \in \mathbb{V}_{2 n}$. However, $\alpha_n 1 \in \mathbb{V}_0$ and $\mathbb{V}_0 \cap \mathbb{V}_{2 n}=0$, since $n \geq 1$, so $\alpha_n 1=0$. Therefore, $\alpha_n=0$. We have shown that $\mathcal{C}^{s}_n=x C^{s}_{n-1} y$.
\begin{coro}
	The following   holds in the $q$-shuffle superalgebra $\mathbb{V}$
\begin{gather}
	\exp \left(\sum_{k=1}^{\infty} -\frac{\{2 k\}_q}{k} x C^{s}_{k-1} y (-t)^k\right)=1+\sum_{k=1}^{\infty} C^{s}_k t^k.
	\label{exp}
\end{gather}
\end{coro}
It is important to emphasize that in equation (\ref{exp}), the exponential function is defined with respect to the $q$-shuffle super-product. Moreover, the notation $x C^{s}_{k-1} y$ denotes the free product.
\begin{proof}
Apply the map $\varphi $ to each side of (\ref{Beck}), and evaluate the result using (\ref{im cor}), (\ref{Beck cor}).
\end{proof}
Formula (\ref{exp}) clearly demonstrates that the elements within the set $\left\{x C^{s}_n y\right\}_{n=0}^{\infty}$  and those in the set $\left\{C^{s}_n\right\}_{n=1}^{\infty}$ can be derived from one another in a recursive manner.
\begin{coro}
		If we define a linear order
		\begin{gather*}
			x<x C^{s}_1<x C^{s}_2<\cdots<x y<x C^{s}_1 y<x C^{s}_2 y<\cdots<C^{s}_2 y<C^{s}_1 y<y,
		\end{gather*}
	then a $PBW$ basis of $U_q^{+}(C(2)^{(2)})$ is obtained by the elements
\begin{gather*}
	\left\{x C^{s}_n\right\}_{n \in \mathbb{N}}, \quad\left\{C^{s}_n y\right\}_{n \in \mathbb{N}}, \quad\left\{x C^{s}_n y\right\}_{n \in \mathbb{N}}.
\end{gather*}
\end{coro}
\begin{coro}
	The elements $\left\{x C^{s}_n y\right\}_{n \in \mathbb{N}}$ mutually commute with respect to the q-shuffle superproduct.
	\end{coro}
\begin{proof}
	Use the first assertion in Lemma \ref{com rel real and im}, together with (\ref{exp}).
\end{proof}
\subsection{Some relations involving the Catalan elements coming from Theorems \ref{Main thm 1} \& \ref{Main thm 2}}

In subsections 3.3 and 3.4, we have completed the proofs of Theorems \ref{Main thm 1} and \ref{Main thm 2}, thereby establishing the correspondence between quantum root vectors and Catalan words. To obtain more intricate commutator relations among the sets $\left\{x C_n^s\right\}_{n=0}^{\infty},\left\{C_n^s y\right\}_{n=0}^{\infty}$, and $\left\{C_n^s\right\}_{n=1}^{\infty}$, we require certain relations within the superalgebra $U_q^{+}(C(2)^{(2)})$.

By Theorems \ref{Main thm 1} and \ref{Main thm 2}, together with Lemmas \ref{R I}, \ref{R R}, \ref{R R 2} \&  \ref{R R 3}, we get
\begin{lemma} \label{R I}
For $i \geq 1$ and $j \geq 0$, the following   relations hold in $U_q^{+}(C(2)^{(2)})$, namely,
\begin{equation}
\begin{aligned}
	E_{i \delta} E_{j \delta+\alpha_0}=\, & E_{j \delta+\alpha_0} E_{i \delta}+(-1)^i q^{2-2 i}\left(q-q^{-1}\right) E_{(i+j) \delta+\alpha_0} \\
	& -q^2\left(q^2-q^{-2}\right) \sum_{\ell=1}^{i-1}(-1)^{\ell} q^{-2 \ell} E_{(j+\ell) \delta+\alpha_0} E_{(i-\ell) \delta},\\
\end{aligned}
\label{com rel real and im 1}
\end{equation}
\begin{equation}
\begin{aligned}
	E_{j \delta+\alpha_1} E_{i \delta}=\, & E_{i \delta} E_{j \delta+\alpha_1}+(-1)^i q^{2-2 i}\left(q-q^{-1}\right) E_{(i+j) \delta+\alpha_1} \\
	& -q^2\left(q^2-q^{-2}\right) \sum_{\ell=1}^{i-1}(-1)^{\ell} q^{-2 \ell} E_{(i-\ell) \delta} E_{(j+\ell) \delta+\alpha_1} .
\end{aligned}
\label{com rel real and im 2}
\end{equation}
\end{lemma}
The corresponding Catalan version is the following
\begin{coro}
	For $i, j \in \mathbb{N}$, the following   relations hold in $\mathbb{V}$,
	\begin{gather*}
		\frac{\left(x C^{s}_i\right) \star C^{s}_j-C^{s}_j \star\left(x C^{s}_i\right)}{q^2-q^{-2}}=\sum_{\ell=1}^j (-1)^{\ell}q^{2-2 \ell}\left(x C^{s}_{i+\ell}\right) \star C^{s}_{j-\ell},\\
		\frac{C^{s}_j \star\left(C^{s}_i y\right)-\left(C^{s}_i y\right) \star C^{s}_j}{q^2-q^{-2}}=\sum_{\ell=1}^j(-1)^{\ell} q^{2-2 \ell} C^{s}_{j-\ell} \star\left(C^{s}_{i+\ell} y\right).
	\end{gather*}
\end{coro}
\begin{lemma} \label{R R}
	For $i>j \geq 0$, the following   relations hold in $U_q^{+}(C(2)^{(2)})$, namely,

	\text{\rm(i)} Assume that $i-j=2 r+1$ is odd, then
\begin{gather}
	E_{i \delta+\alpha_0} E_{j \delta+\alpha_0}=-q^{-2} E_{j \delta+\alpha_0} E_{i \delta+\alpha_0}+
	\left(q^2-q^{-2}\right) \sum_{\ell=1}^r(-1)^{\ell} q^{-2 \ell} E_{(j+\ell) \delta+\alpha_0} E_{(i-\ell) \delta+\alpha_0},
	\label{com rel real and real 1}
\end{gather}
\begin{gather}
	E_{j \delta+\alpha_1} E_{i \delta+\alpha_1}=-q^{-2} E_{i \delta+\alpha_1} E_{j \delta+\alpha_1}+\left(q^2-q^{-2}\right) \sum_{\ell=1}^r(-1)^{\ell} q^{-2 \ell} E_{(i-\ell) \delta+\alpha_1} E_{(j+\ell) \delta+\alpha_1}.
	\label{com rel real and real 2}
\end{gather}

	\text{\rm(ii)} Assume that $i-j=2 r$ is even, then
\begin{equation}
	\begin{aligned}
		E_{i \delta+\alpha_0} E_{j \delta+\alpha_0}=-q^{-2} E_{j \delta+\alpha_0} & E_{i \delta+\alpha_0}+(-1)^{r+1}q^{j-i+1}\left(q+q^{-1}\right) E_{(r+j) \delta+\alpha_0}^2 \\
		& +\left(q^2-q^{-2}\right) \sum_{\ell=1}^{r-1}(-1)^{\ell} q^{-2 \ell} E_{(j+\ell) \delta+\alpha_0} E_{(i-\ell) \delta+\alpha_0},
	\end{aligned}
\label{com rel real and real 3}
\end{equation}
\begin{equation}
	\begin{aligned}
		E_{j \delta+\alpha_1} E_{i \delta+\alpha_1}=-q^{-2} E_{i \delta+\alpha_1} & E_{j \delta+\alpha_1}+(-1)^{r+1}q^{j-i+1}\left(q+q^{-1}\right) E_{(r+j) \delta+\alpha_1}^2 \\
		& +\left(q^2-q^{-2}\right) \sum_{\ell=1}^{r-1}(-1)^{\ell} q^{-2 \ell} E_{(i-\ell) \delta+\alpha_1} E_{(j+\ell) \delta+\alpha_1} .
	\end{aligned}
\label{com rel real and real 4}
\end{equation}
\end{lemma}
The corresponding Catalan version is the following
\begin{coro}
	For $i>j \geq 0$, the following   relations hold in $\mathbb{V}$, namely,
	
	\text{\rm(i)} Assume that $i-j=2 r+1$ is odd, then
	\begin{gather*}
		\frac{q\left(x C^{s}_i\right) \star\left(x C^{s}_j\right)+q^{-1}\left(x C^{s}_j\right) \star\left(x C^{s}_i\right)}{q^2-q^{-2}}=\sum_{\ell=1}^r (-1)^{\ell}q^{1-2 \ell}\left(x C^{s}_{j+\ell}\right) \star\left(x C^{s}_{i-\ell}\right),\\
		\frac{q\left(C^{s}_j y\right) \star\left(C^{s}_i y\right)+q^{-1}\left(C^{s}_i y\right) \star\left(C^{s}_j y\right)}{q^2-q^{-2}}=\sum_{\ell=1}^r(-1)^{\ell} q^{1-2 \ell}\left(C^{s}_{i-\ell} y\right) \star\left(C^{s}_{j+\ell} y\right);
	\end{gather*}
	
	\text{\rm(ii)} Assume that $i-j=2 r$ is even, then
	\begin{equation*}
		\begin{gathered}
			\frac{q\left(x C^{s}_i\right) \star\left(x C^{s}_j\right)+q^{-1}\left(x C^{s}_j\right) \star\left(x C^{s}_i\right)}{q^2-q^{-2}}+(-1)^{r}\frac{q^{j-i+2}\left(x C^{s}_{j+r}\right) \star\left(x C^{s}_{i-r}\right)}{q-q^{-1}} \\
			=\sum_{\ell=1}^{r-1}(-1)^{\ell} q^{1-2 \ell}\left(x C^{s}_{j+\ell}\right) \star\left(x C^{s}_{i-\ell}\right),
		\end{gathered}
	\end{equation*}
	\begin{equation*}
		\begin{gathered}
			\frac{q\left(C^{s}_j y\right) \star\left(C^{s}_i y\right)+q^{-1}\left(C^{s}_i y\right) \star\left(C^{s}_j y\right)}{q^2-q^{-2}}+(-1)^{r}\frac{q^{j-i+2}\left(C^{s}_{i-r} y\right) \star\left(C^{s}_{j+r} y\right)}{q-q^{-1}} \\
			=\sum_{\ell=1}^{r-1}(-1)^{\ell} q^{1-2 \ell}\left(C^{s}_{i-\ell} y\right) \star\left(C^{s}_{j+\ell} y\right).
		\end{gathered}
	\end{equation*}
\end{coro}

\begin{lemma} \label{R R 2}
For $i, j \in \mathbb{N}$, the following   relations hold in $U_q^{+}(C(2)^{(2)})$,
\begin{gather}
E_{i \delta+\alpha_0} E_{(j+1) \delta}-E_{(j+1) \delta} E_{i \delta+\alpha_0}=-q^2 E_{(i+1) \delta+\alpha_0} E_{j \delta}+q^{-2} E_{j \delta} E_{(i+1) \delta+\alpha_0},
\label{com rel real and im 3}
\end{gather}
\begin{gather}
E_{(j+1) \delta} E_{i \delta+\alpha_1}-E_{i \delta+\alpha_1} E_{(j+1) \delta}=-q^2 E_{j \delta} E_{(i+1) \delta+\alpha_1}+q^{-2} E_{(i+1) \delta+\alpha_1} E_{j \delta} .
\label{com rel real and im 4}
\end{gather}
\end{lemma}
The corresponding Catalan version is the following
\begin{coro}
	For $i, j \in \mathbb{N}$, the following   relations hold in $\mathbb{V}$,
	\begin{gather*}
		\left(x C^{s}_i\right) \star C^{s}_{j+1}-C^{s}_{j+1} \star\left(x C^{s}_i\right)=-q^2\left(x C^{s}_{i+1}\right) \star C^{s}_j+q^{-2} C^{s}_j \star\left(x C^{s}_{i+1}\right),\\
		C^{s}_{j+1} \star\left(C^{s}_i y\right)-\left(C^{s}_i y\right) \star C^{s}_{j+1}=-q^2 C^{s}_j \star\left(C^{s}_{i+1} y\right)+q^{-2}\left(C^{s}_{i+1} y\right) \star C^{s}_j .
	\end{gather*}
\end{coro}
\begin{lemma} \label{R R 3}
	The following   relations hold in $U_q^{+}(C(2)^{(2)})$, namely,
\begin{equation}
	\begin{aligned}
		& q E_{(i+1) \delta+\alpha_0} E_{i \delta+\alpha_0}=-q^{-1} E_{i \delta+\alpha_0} E_{(i+1) \delta+\alpha_0}, \\
		& q E_{i \delta+\alpha_1} E_{(i+1) \delta+\alpha_1}=-q^{-1} E_{(i+1) \delta+\alpha_1} E_{i \delta+\alpha_1}.
	\end{aligned}
\label{q real and real 1}
\end{equation}
For distinct $i, j \in \mathbb{N}$, one has
\begin{equation}
\begin{aligned}
	q E_{(i+1) \delta+\alpha_0} E_{j \delta+\alpha_0} & +q^{-1} E_{j \delta+\alpha_0} E_{(i+1) \delta+\alpha_0} \\
	& =-q^{-1} E_{i \delta+\alpha_0} E_{(j+1) \delta+\alpha_0}-q E_{(j+1) \delta+\alpha_0} E_{i \delta+\alpha_0},
\end{aligned}
\label{q real and real 2}
\end{equation}
\begin{equation}
\begin{aligned}
	q E_{j \delta+\alpha_1} E_{(i+1) \delta+\alpha_1} & +q^{-1} E_{(i+1) \delta+\alpha_1} E_{j \delta+\alpha_1} \\
	& =-q^{-1} E_{(j+1) \delta+\alpha_1} E_{i \delta+\alpha_1}-q E_{i \delta+\alpha_1} E_{(j+1) \delta+\alpha_1}.
\end{aligned}
\label{q real and rea 3}
\end{equation}
\end{lemma}
The corresponding Catalan version is the following
\begin{coro} The following   relations are true in $\mathbb V$\,$:$
\begin{gather*}
	q\left(x C^{s}_{i+1}\right) \star\left(x C^{s}_i\right)=-q^{-1}\left(x C^{s}_i\right) \star\left(x C^{s}_{i+1}\right),\\
	q\left(C^{s}_i y\right) \star\left(C^{s}_{i+1} y\right)=-q^{-1}\left(C^{s}_{i+1} y\right) \star\left(C^{s}_i y\right).
\end{gather*}
	For distinct $i, j \in \mathbb{N}$, one has
\begin{gather*}
	q\left(x C^{s}_{i+1}\right) \star\left(x C^{s}_j\right)+q^{-1}\left(x C^{s}_j\right) \star\left(x C^{s}_{i+1}\right)=-q^{-1}\left(x C^{s}_i\right) \star\left(x C^{s}_{j+1}\right)-q\left(x C^{s}_{j+1}\right) \star\left(x C^{s}_i\right),\\
	q\left(C^{s}_j y\right) \star\left(C^{s}_{i+1} y\right)+q^{-1}\left(C^{s}_{i+1} y\right) \star\left(C^{s}_j y\right)=-q^{-1}\left(C^{s}_{j+1} y\right) \star\left(C^{s}_i y\right)-q\left(C^{s}_i y\right) \star\left(C^{s}_{j+1} y\right).
\end{gather*}
\end{coro}
\begin{coro}
	For $k, \ell \in \mathbb{N}$, one has
\begin{gather}
		\label{xCk+l+1}
	x C^{s}_{k+\ell+1}=\frac{\left(x C^{s}_{\ell}\right) \star\left(x C^{s}_k y\right)-\left(x C^{s}_k y\right) \star\left(x C^{s}_{\ell}\right)}{q+q^{-1}},\\
	\label{Ck+l+1y}
	C^{s}_{k+\ell+1} y=\frac{\left(x C^{s}_k y\right) \star\left(C^{s}_{\ell} y\right)-\left(C^{s}_{\ell} y\right) \star\left(x C^{s}_k y\right)}{q-q^{-1}} .
\end{gather}
\end{coro}

\subsection{Equitable  presentation of a minimal bosonization of $U_q(C(2)^{(2)})$}
According to Definition 8 of the quantum affine super-algebra $U_q(C(2)^{(2)})$ (\cite{KLT}), we present its bosonization by the grade involution $P$.

It should be noted that the $\mathbb{Z}_2$-grading of $U_q(C(2)^{(2)})$ can be explicitly realized. This is accomplished by introducing the grade involution $P$ into the set of generators and stipulating that the generators of even-degree commute with $P$, whereas the generators of odd-degree anticommute with it. As a result, the quantum affine superalgebra  $U_q(C(2)^{(2)})$ can be bosonized into an unital associative $\mathbb K$-algebra. This algebra is generated by the elements $k_{\mathrm{d}}^{ \pm 1}:=q^{ \pm h_{\mathrm{d}}}, k_\alpha^{ \pm 1}:=q^{ \pm h_\alpha}, k_{\delta-\alpha}^{ \pm 1}:=q^{ \pm h_{\delta-\alpha}}, e_{ \pm \alpha}, e_{ \pm(\delta-a)}$ and the involution $P$, all of which adhere to the additional relations:
\begin{gather*}
\left\{P, e_{ \pm \alpha}\right\}=0,\quad \left\{P, e_{ \pm(\delta-a)}\right\}=0,\quad \left[P, k_\alpha^{ \pm 1}\right]=0,\quad \left[P, k_{\delta-\alpha}^{ \pm 1}\right]=0,\quad P^2=1.
\end{gather*}

We obtain an equitable presentation for a minimal bosonization of the quantum affine superalgebra $U_q(C(2)^{(2)})$, which is given below.
 \begin{proposition}The bosonization of the quantum affine superalgebra $U_q(C(2)^{(2)})$ is isomorphic to the unital associative $\mathbb{K}$-algebra with generators $X_i,\; X_i^{-1},\; Y_i, \;Z_i$ $(i=0,\; 1)$, $\Omega$ and the following relations
$$
\begin{aligned}
	& X_i X_i^{-1}=X_i^{-1} X_i=1, \quad \Omega^2=1, \\
	& \frac{q Y_i X_i+q^{-1} X_i Y_i}{q+q^{-1}}=1, \\
	& \frac{q X_i Z_i+q^{-1} Z_i X_i}{q+q^{-1}}=1, \\
	& \frac{q Y_i Z_i+q^{-1} Z_i Y_i}{q+q^{-1}}=1, \\
	& \frac{q Y_i Z_j+q^{-1} Z_j Y_i}{q+q^{-1}}=1, \\
	& Y_i^3 Y_j-\{3\}_q Y_i^2 Y_j Y_i+\{3\}_q Y_i Y_j Y_i^2-Y_j Y_i^3=0, \quad i \neq j, \\
	& Z_i^3 Z_j-\{3\}_q Z_i^2 Z_j Z_i+\{3\}_q Z_i Z_j Z_i^2-Z_j Z_i^3=0, \quad i \neq j,
\end{aligned}
$$
together with the following   relations related to $P:$
\begin{gather*}
X_i \Omega+\Omega X_i=2 X_i \Omega, \quad Y_i \Omega+\Omega Y_i=2 X_i^{-1} \Omega, \quad Z_i \Omega+\Omega Z_i=2 X_i^{-1} \Omega.
\end{gather*}
An isomorphism with the bosonization is given by
$$
\begin{aligned}
 X_i^{ \pm} 	&\mapsto k_i^{ \pm} P, \\
\Omega 	& \mapsto P, \\
 Y_i 	&\mapsto k_i^{-1} P-q\left(q+q^{-1}\right) e_i^{+} k_i^{-1} P, \\
 Z_i 	&\mapsto k_i^{-1} P+\left(q+q^{-1}\right) e_i^{-} P.
\end{aligned}
$$
The inverse of this isomorphism is given by
$$
\begin{aligned}
	k_i^ \pm &\mapsto X_i^ \pm \Omega , \hfill \\
P 	&\mapsto \Omega , \hfill \\
	e_i^ +  &\mapsto {q^{ - 1}}{(q +{q^{ - 1}})^{ - 1}}\left( {1 - {Y_i}{X_i}} \right), \hfill \\
	e_i^ -  &\mapsto {(q + {q^{ - 1}})^{ - 1}}\left( {{Z_i}\Omega  - {X_i}^{ - 1}\Omega } \right). \hfill \\
\end{aligned}
$$
 \end{proposition}
\begin{proof}
 One readily checks that each map is a homomorphism of $\mathbb{K}$-algebras and that the maps are inverses. It follows that each map is an isomorphism of $\mathbb{K}$-algebras.
\end{proof}
We will also just list Definition \ref{iquantum2} and Theorem \ref{embedding2}, the details will appear in \cite{ZH1}.
 \begin{defi}\label{iquantum2}
 	Let $\mathbb K=\mathbb{Q}(q)$ denote the field of rational functions in an indeterminate $q$ and let $c \in \mathbb{Q}(q)$ such that $c(1)=1$. The $q$-Onsager superalgebra $\mathcal{B}^{s}_c$ is the unital $\mathbb{K}$-algebra generated by two elements $B_0, B_1$ with the defining relations
 $$
 \begin{aligned}
 	\sum_{i=0}^3(-1)^i\left\{\begin{array}{l}
 		3 \\
 		i
 	\end{array}\right\}_q B_0^{3-i} B_1 B_0^i & =q c\left(q-q^{-1}\right)^2\left(B_0 B_1-B_1 B_0\right), \\
 	\sum_{i=0}^3(-1)^i\left\{\begin{array}{l}
 		3 \\
 		i
 	\end{array}\right\}_q B_1{ }^{3-i} B_0 B_1{ }^i & =q c\left(q-q^{-1}\right)^2\left(B_1 B_0-B_0 B_1\right).
 \end{aligned}
 $$
 \end{defi}	
 \begin{theorem}\label{embedding2}
 	There exists a superalgebra embedding
 	\begin{gather*}
 		\iota: \mathcal{B}^{s}_c \rightarrow U_q(C(2)^{(2)}) \quad \text { with } \quad \iota\left(B_i\right)=F_i-c E_i K_i^{-1}, \quad \text { for } i \in\{0,1\}.
 	\end{gather*}
 \end{theorem}

\begin{remark}
	We can also do the similar thing for $U_q[\operatorname{osp}(1|2)^{(1)}]$,  and construct the alternating $PBW$ basis in our next and third paper $($\cite{ZH1, ZH2}$)$.
\end{remark}
	\section*{\textbf{ACKNOWLEDGMENTS}}
The paper is supported by the NNSFC (Grant No.
12171155) and in part by the Science and Technology Commission of Shanghai Municipality (No. 22DZ2229014)

\section*{\textbf{AUTHOR DECLARATIONS}}
\subsection*{\textbf{Conflict of Interest}}
The authors have no conflicts to disclose.
\subsection*{\textbf{Author Contributions}}
\textbf{Xin Zhong}: Conceptualization (equal) Formal analysis (equal) Methodology (equal) Writing - original draft (equal).
\textbf{Naihong Hu}: Supervision (equal) Validation (equal) Writing - review and editing (equal).

\section*{\textbf{DATA AVAILABILITY}}
All data that support the findings of this study are included within the article.

\end{document}